\documentclass[a4paper,11pt]{amsart}
\usepackage{amsmath,amssymb,url}
\usepackage[active]{srcltx}
\usepackage[T1]{fontenc}
\usepackage[latin1]{inputenc}
\usepackage{xcolor}
\usepackage{hyperref}
\usepackage{latexsym,amssymb}
\usepackage{stmaryrd}
\usepackage{mathtools}
\usepackage{tikz}
\usetikzlibrary{arrows,automata,positioning}

\numberwithin{equation}{section}

\renewcommand{\geq}{\geqslant}

\renewcommand{\ge}{\geqslant}

\let\cal=\mathcal
\def\Cl#1{\ensuremath{\cal{#1}}}
\newcommand{\cont}{\mathcal{C}}

\theoremstyle{plain}
\newtheorem{Thm}{Theorem}[section]
\newtheorem{Prop}[Thm]{Proposition}
\newtheorem{Lemma}[Thm]{Lemma}

\newtheorem{Cor}[Thm]{Corollary}

\theoremstyle{definition}

\newtheorem{Remark}[Thm]{Remark}
\newtheorem{eg}[Thm]{Example}

\begin{document}

\title{What makes a Stone topological algebra profinite} \thanks{The
  first author acknowledges partial support by CMUP
  (UID/MAT/00144/2020) which is funded by FCT (Portugal) with national
  (MATT'S) and European structural funds (FEDER) under the partnership
  agreement PT2020. %
  The work was carried out in part at Masaryk University, whose
  hospitality is gratefully acknowledged, with the support of the FCT
  sabbatical scholarship SFRH/BSAB/142872/2018. %
  The second author is grateful for the financial support provided by
  the Centre for Mathematics of the University of Coimbra
  (UIDB/00324/2020, funded by the Portuguese Government through
  FCT/MCTES), the Centre for Mathematics of the University of Porto
  (UIDB/00144/2020, funded by the Portuguese Government through
  FCT/MCTES), as well as a PhD grant from FCT/MCTES
  (PD/BD/150350/2019). %
  The third author was supported by Grant 19-12790S of the Grant
  Agency of the Czech Republic.}

\author[J. Almeida]{Jorge Almeida}%
\address{CMUP, Dep.\ Matem\'atica, Faculdade de Ci\^encias,
  Universidade do Porto, Rua do Campo Alegre 687, 4169-007 Porto,
  Portugal}
\email{jalmeida@fc.up.pt}

\author[H. Goulet-Ouellet]{Herman Goulet-Ouellet}%
\address{University of Coimbra, CMUC, Department of Mathematics,
  Apartado 3008, EC Santa Cruz,
  3001-501 Coimbra, Portugal}%
\email{hgouletouellet@student.uc.pt}

\author[O. Kl\'ima]{Ond\v rej Kl\'ima}%
\address{Dept.\ of Mathematics and Statistics, Masaryk University,
  Kotl\'a\v rsk\'a 2, 611 37 Brno, Czech Republic}%
\email{klima@math.muni.cz}


\makeatletter
\@namedef{subjclassname@2020}{%
  \textup{2020} Mathematics Subject Classification}
\makeatother

\begin{abstract}
  This paper is a contribution to understanding what properties should
  a topological algebra on a Stone space satisfy to be profinite. We
  reformulate and simplify proofs for some known properties using
  syntactic congruences. We also clarify the role of various
  alternative ways of describing syntactic congruences, namely by
  finite sets of terms and by compact sets of continuous self mappings
  of the algebra.
\end{abstract}

\maketitle

\section{Introduction}
\label{sec:into}

Profinite algebras, that is, inverse limits of inverse systems of
finite algebras, appear naturally in several contexts. There is an
extensive theory of profinite groups, which appear as Galois groups
but are also studied as a generalization of finite groups
\cite{Fried&Jarden:2008,Ribes&Zalesskii:2010}. One may also view
$p$-adic number theory as an early study of special profinite
algebraic structures~\cite{Robert:2000}. In the context of general
algebras, profinite topologies seem to have first appeared
in~\cite{Birkhoff:1937}. A fruitful line of development came about
with the discovery that formal equalities between elements of free
profinite algebras may be used to describe pseudovarieties
\cite{Reiterman:1982,Banaschewski:1983}, which are classes of finite
algebras of a fixed type closed under taking homomorphic images,
subalgebras, and finite direct products. Since pseudovarieties play an
important role in algebraic theories developed for computer science,
profinite algebras are also a useful tool in that context. In
particular, profinite semigroups have been extensively used (see, for
instance, \cite{Almeida:1994a, Weil:2002, Almeida&Volkov:2001a,
  Almeida:2003cshort, Rhodes&Steinberg:2009qt, Almeida&Costa:2015hb,
  Almeida&ACosta&Kyriakoglou&Perrin:2020b}).

As the underlying topological structure of a profinite algebra is
compact and 0-dimensional, that is, a Stone space, profinite algebras
may also be viewed as dual spaces of Boolean algebras. For finitely
generated relatively free profinite algebras, the corresponding
Boolean algebras have special significance
\cite[Theorem~3.6.1]{Almeida:1994a} and are particularly relevant in
the applications of finite semigroup theory to the theory of formal
languages, where they appear as Boolean algebras of regular languages.
The dual role of the algebraic operations in relatively free profinite
algebras has also been investigated
\cite{Gehrke&Grigorieff&Pin:2008,Gehrke&Grigorieff&Pin:2010,Gehrke:2016a}.

For certain classes of (topological) algebras, it turns out that being
a Stone space is sufficient to guarantee profiniteness. Special cases
were considered in~\cite{Numakura:1957} but the essential ingredient
lies in the fact that syntactic congruences are determined by finitely
many terms \cite{Almeida:1989b, Clark&Davey&Freese&Jackson:2004} using
an idea of Hunter \cite{Hunter:1988} that may be traced back to
Numakura~\cite{Numakura:1957}. For such classes of algebras,
profiniteness is thus a purely topological property, although this is
not true in general.

Recently, several characterizations of profiniteness in a Stone
topological algebra $A$ have been obtained
in~\cite{Schneider&Zumbragel:2017}. They are formulated in terms of
topological properties of the translation monoid of the algebra $A$,
which is a submonoid of the monoid of continuous transformations
of~$A$, which is itself a topological monoid under the compact-open
topology.

We explore further the role of syntactic congruences in the
characterization of profiniteness. This leads to an extension to
topological algebras over topological signatures of Gehrke's
sufficient condition for a quotient of a profinite algebra to be
profinite, with a simplified proof. The proofs of Schneider and
Zumbr\"agel's characterizations of profiniteness in this language are
also somewhat simplified. The key ingredient is quite simple: a Stone
topological algebra is profinite if and only if the syntactic
congruence of every clopen subset is clopen.

As already mentioned above, the existence of descriptions of the
syntactic congruences of clopen subsets by a finite number of terms is
a sufficient condition for profiniteness. In such terms, all but one
variable are evaluated to arbitrary values in the algebra, which means
that potentially infinitely many polynomials in one variable are used.
In a profinite algebra, each such syntactic congruence may in fact be
described by finitely many linear polynomials in one variable, which
may depend on the congruence. We explore more generally the property
of a syntactic congruence being determined by a compact set of
continuous self mappings of the algebra, a property that, for a clopen
subset of a locally compact algebra is equivalent to the congruence
being clopen. This leads to several further characterizations of
profiniteness for a Stone topological algebra in terms of how their
syntactic congruences may be described.

\section{Topological algebras}
\label{sec:top-algebras}

We say that an equivalence relation on a topological space $X$ is
\emph{closed} (respectively \emph{open} or \emph{clopen}) if it is a
closed (respectively open or clopen) subset of the product space
$X\times X$. It is easy to verify that an equivalence relation is open if
and only if it is clopen, if and only if its classes are open, if and
only if its classes are clopen (see, for instance,
\cite[Exercise~3.40]{Almeida&ACosta&Kyriakoglou&Perrin:2020b}). For a
closed equivalence relation, the classes are closed, but the converse
fails in general
\cite[Exercise~3.39]{Almeida&ACosta&Kyriakoglou&Perrin:2020b}. Given
an equivalence relation $\theta$ on a topological space~$X$, the
quotient set $X/\theta$ is endowed with the largest topology that
renders continuous the natural mapping $X\to X/\theta$. In particular,
a set of $\theta$-classes is closed (respectively, open) if and only
if so is its union in~$X$.

Unlike some literature on topology, we require the Hausdorff
separation property for a space to be locally compact or compact.

The following observation may be considered as an exercise in
topology, based on \cite[Chapter~I, \S10.4,
Proposition~8]{Bourbaki:1998GTa}.

\begin{Prop}
  \label{p:quotient-Hausdorff}
  Let $X$ be a compact space and $\theta$ an equivalence relation
  on~$X$. Then the quotient space $X/\theta$ is compact if and only if
  $\theta$ is closed.
\end{Prop}

Following \cite{Schneider&Zumbragel:2017}, by a \emph{signature} we
mean a sequence $\Omega=(\Omega_n)_{n\in\mathbb{N}}$ of sets. We say
that it is a \emph{topological signature} if each set $\Omega_n$ is
endowed with a topology.

An \emph{$\Omega$-algebra} is a pair $(A,E)$, where $A$ is a set and
$E=(E_n^A)_{n\in\mathbb{N}}$ is a sequence of \emph{evaluation}
mappings $E_n^A:\Omega_n\times A^n\to A$. In case $A$ is a topological
space and $\Omega$ is a topological signature, we say that the
$\Omega$-algebra is a \emph{topological $\Omega$-algebra} if each
mapping $E_n^A$ is continuous. For an $\Omega$-algebra $(A,E)$ and
$w\in\Omega_n$, we let $w_A:A^n\to A$ be the operation defined by
\begin{displaymath}
  w_A(a_1,\ldots,a_n)=E_n^A(w,a_1,\ldots,a_n).
\end{displaymath}

Reference to the sequences $E$ and $\Omega$, which should be
understood from the context, will usually be omitted and so we simply
say that $A$ is an algebra.

We view finite algebras as discrete topological algebras. Note that
the requirement that the evaluation mappings be continuous may still
be nontrivial. For instance, if we take the signature $\Omega$ with
one binary operation symbol and $\Omega_1$ the one point ($\infty$)
compactification of~$\mathbb{N}$, we may consider finite semigroups as
$\Omega$-algebras by interpreting the binary operation symbol as the
semigroup multiplication, each unary operation symbol $n\in\mathbb{N}$
as the $n!$ power and $\infty$ as the unique idempotent power.
The evaluation mappings are continuous, and so we obtain a topological
$\Omega$-algebra. However, it is sometimes useful to consider an
(``unnatural'') interpretation of the operation symbol $\infty$ that
makes the evaluation mapping $E_1$ discontinuous. This idea underlies
the recent paper~\cite{Almeida&Costa&Zeitoun:2016}.

For a class \Cl K of topological algebras, we say that a topological
algebra $A$ is \emph{residually \Cl K} if, for every pair $a,a'$ of
distinct elements of~$A$, there is a continuous homomorphism
$\varphi:A\to B$ into a member $B$ of~\Cl K such that
$\varphi(a)\ne\varphi(b)$.

A topological algebra $A$ is said to be \emph{profinite} if it is
residually finite and $A$ is compact. Equivalently, $A$ is an inverse
limit of finite discrete algebras (see, for instance,
\cite{Almeida&Costa:2015hb}).

A continuous mapping from a topological space $X$ to a topological
algebra $A$ is said to be a \emph{generating mapping} if its image
generates (algebraically) a dense subalgebra of~$A$. In case $X$ is a
subset of~$A$, we say that the topological algebra $A$ is
\emph{generated} by~$X$ if the inclusion $X\hookrightarrow A$ is a
generating mapping. We also say that $A$ is \emph{$X$-generated} if
there is a generating mapping $X\to A$ and that it is \emph{finitely
  generated} if it is $X$-generated for some finite set $X$.

By a \emph{congruence} on an $\Omega$-algebra $A$ we mean an
equivalence relation $\theta$ that is compatible with the $\Omega$
operations. The set $A/\theta$ of all congruence classes $a/\theta$
inherits a natural structure of $\Omega$-algebra: for $w\in\Omega_n$,
$w_{A/\theta}(a_1/\theta,\ldots,a_n/\theta)=w_A(a_1,\ldots,a_n)/\theta$.
In case $A$ is a topological $\Omega$-algebra, the quotient algebra
$A/\theta$ is a topological $\Omega$-algebra for the quotient
topology.

For a set $X$, let \emph{$T_\Omega(X)$} be the \emph{$\Omega$-term
  algebra}, that is, the absolutely free $\Omega$-algebra. The
elements of $T_\Omega(X)$ are usually viewed in computer science as
trees whose leaves are labeled by elements of~$X$ or of~$\Omega_0$ and
each non-leaf node is labeled by an element of some~$\Omega_n$, in
which case the node has exactly $n$ sons.

For instance, for $u\in\Omega_2$, $v\in\Omega_3$, and $w\in\Omega_0$,
the term
\begin{displaymath}
  u(v(x_1,u(w,x_1),x_3),u(x_3,x_2))
\end{displaymath}
is represented by the labeled tree pictured below.
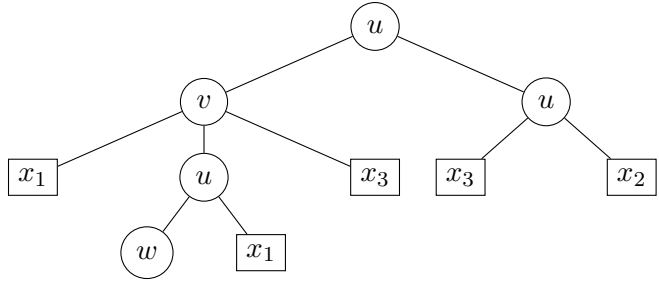
\begin{center}
  \begin{tikzpicture}[level/.style={sibling distance = 9cm/#1,
      level distance = 20mm},
    treenode/.style = {align=center, text centered},
    root/.style = {treenode, shape=circle, draw},
    op/.style = {treenode, shape=circle, draw},
    leaf/.style = {treenode, draw},
    scale=.5]
    \node [root] {$u$}
    child {node [op] {$v$}
      child {node [leaf] {$x_1$}}
      child {node [op] {$u$}
        child {node [op] {$w$}}
        child {node [leaf] {$x_1$}}
      }
      child {node [leaf] {$x_3$}}
    }
    child {node [op] {$u$}
      child {node [leaf] {$x_3$}}
      child {node [leaf] {$x_2$}}
    };
  \end{tikzpicture}
\end{center}
If the tree representing the term $t$ has exactly one occurrence of
$x\in X$ as a leaf label, then we say that $t$ is \emph{linear
  in}~$x$. For instance, the term above is linear only in~$x_2$.

Let $\varphi:X\to A$ be any mapping. Since the term algebra
$T_\Omega(X)$ is the absolutely free algebra over the set $X$, there
is a unique homomorphism $\hat{\varphi}:T_\Omega(X)\to A$ such that
$\hat{\varphi}\circ\iota=\varphi$, where $\iota$ is the inclusion
mapping of $X$ in~$T_\Omega(X)$. In case $X=\{x_1,\ldots,x_m\}$, we
also denote $\hat{\varphi}(t)$ by
$t_A\bigl(\varphi(x_1),\ldots,\varphi(x_m)\bigr)$ for $t\in
T_\Omega(X)$. In this case, given elements $a_j\in A$ ($j\ne i$), a
term $t\in T_\Omega(X)$ determines a \emph{polynomial transformation}
of $A$ given by
\begin{displaymath}
  a\mapsto t_A(a_1,\ldots,a_{i-1},a,a_{i+1},\ldots,a_m).
\end{displaymath}
Note that, when $A$ is a topological algebra, such transformations are
continuous. The \emph{translation monoid} $M(A)$ consists of all such
polynomial transformations given by terms that are linear in the
distinguished variable $x_i$.

More generally, for two topological spaces $X$ and $Y$, let
$\cont(X,Y)$ be the set of continuous mappings from $X$\ to $Y$; it is
endowed with the compact-open topology, for which a subbase consists
of all sets of the form
\begin{displaymath}
  [K,U]=\bigl\{f\in\Cl C(X,Y): f(K)\subseteq U\bigr\},
\end{displaymath}
where $K\subseteq X$ and $U\subseteq Y$ are respectively compact and
open. Subsets of $\cont(X,Y)$ are endowed with the induced topology.
We also write $\cont(X)$ for $\cont(X,X)$. Note that, in case $A$ is a
topological algebra, $M(A)$ is a submonoid of the monoid $\Cl C(A)$.

Note that, if $X$ is compact 0-dimensional (that is, a \emph{Stone
  space}), then we may restrict the choice of both $K$ and $U$ to be
clopen subsets of~$X$ in the subbasic open sets $[K,U]$ of the
compact-open topology of $\cont(X)$ (cf.~\cite[Chapter~X, \S3.4,
Remark~2]{Bourbaki:1998GTb}).

\section{Syntactic congruences}
\label{sec:syntactic-congruences}

We say that an equivalence relation $\theta$ on a set $A$
\emph{saturates} a subset $L$ of~$A$ if $L$ is a union of
$\theta$-classes.
The following purely algebraic result is well known. We provide a
proof for the sake of completeness.

\begin{Lemma}
  \label{l:synt-congr}
  For an algebra $A$ and a subset $L$ of~$A$, the set of all pairs
  $(a,a')\in A\times A$ such that
  \begin{displaymath}
    \forall f\in M(A)\ \bigl(f(a)\in L \iff f(a')\in L\bigr)
  \end{displaymath}
  is the largest congruence on $A$ saturating~$L$.
\end{Lemma}

\begin{proof}
  Let $\theta$ denote the set of pairs in the statement of the lemma.
  Note that it is an equivalence relation on~$A$. Let $w\in\Omega_n$
  and $(a_i,a_i')\in\theta$ for $i=1,\ldots,n$ and suppose that $f\in
  M(A)$ is such that $f\bigl(w_A(a_1,\ldots,a_n)\bigr)\in L$.
  Considering the polynomial transformation $g$ given by
  $g(x)=f\bigl(w_A(x,a_2,\ldots,a_n)\bigr)$, we deduce from
  $(a_1,a_1')\in\theta$ that $f\bigl(w_A(a_1',a_2,\ldots,a_n)\bigr)\in
  L$. Proceeding similarly on each of the remaining components, we see
  that $f\bigl(w_A(a_1',\ldots,a_n')\bigr)\in L$. Hence, $\theta$ is a
  congruence on~$A$. Taking for $f$ the identity transformation of~$A$
  we conclude that $\theta$ saturates~$L$.

  Now, suppose that $\rho$ is a congruence on~$A$ saturating~$L$ and
  let $(a,a')\in\rho$ and $f\in M(A)$. Because polynomial
  transformations of~$A$ preserve $\rho$-equivalence, we have
  $\bigl(f(a),f(a')\bigr)\in\rho$. Since $\rho$ saturates~$L$, it
  follows that $f(a)\in L$ if and only if $f(a')\in L$, which shows
  that $(a,a')\in\theta$. Hence, $\rho$ is contained in~$\theta$,
  thereby completing the proof of the lemma.
\end{proof}

The congruence of the lemma is called the \emph{syntactic congruence}
of~$L$ on~$A$ and it is denoted $\sigma_L^A$.

\begin{Remark}
  \label{rk:synt-congr}
  Let $L$ be a subset of an algebra $A$ and let $\alpha_L$ be the
  equivalence relation whose classes are $L$ and $A\setminus L$. Then
  we may reformulate the definition of the syntactic congruence by the
  following formula:
  \begin{equation}
    \label{eq:synt-congr}
    \sigma_L^A
    =\bigcap_{f\in M(A)}(f\times f)^{-1}(\alpha_L)
    =\bigcap_{f\in M(A)}\alpha_{f^{-1}(L)}.
  \end{equation}
\end{Remark}

\begin{Cor}
  \label{c:synt-congr-closed}
  Let $L$ be a clopen subset of a topological algebra~$A$. Then the
  syntactic congruence $\sigma_L^A$ is closed.
\end{Cor}

\begin{proof}
  Consider the equivalence relation $\alpha_L$ of
  Remark~\ref{rk:synt-congr}, so that formula~\eqref{eq:synt-congr}
  holds. Since $\alpha_L$ is a clopen subset of $A\times A$ and each
  mapping $f\in M(A)$ is continuous, so that $f\times
  f\in\cont(A\times A)$, it follows that $\sigma_L^A$ is an
  intersection of clopen sets of the form $(f\times f)^{-1}(\alpha_L)$
  and, therefore, it is closed.
\end{proof}

The following result is another simple application of
Lemma~\ref{l:synt-congr}.

\begin{Prop}
  \label{p:lift-synt-congr}
  Let $\varphi:A\to B$ be an onto homomorphism between two
  $\Omega$-algebras and let $L$ be a subset of~$B$. Then we have
  \begin{displaymath}
    (\varphi\times\varphi)^{-1}\sigma_L^B=\sigma_{\varphi^{-1}(L)}^A,
  \end{displaymath}
  so that $\varphi$ induces an isomorphism
  $A/\sigma_{\varphi^{-1}(L)}^A\to B/\sigma_L^B$.
\end{Prop}

\begin{proof}
  By the general correspondence theorem of universal algebra
  \cite[Theorem 6.20]{Burris&Sankappanavar:1981},
  $(\varphi\times\varphi)^{-1}\sigma_L^B$ is a congruence on~$A$. It
  saturates $\varphi^{-1}(L)$ since, if
  $(a,a')\in(\varphi\times\varphi)^{-1}\sigma_L^B$ and
  $a\in\varphi^{-1}(L)$, then $(\varphi(a),\varphi(a'))\in\sigma_L^B$
  and $\varphi(a)$ belongs to~$L$ and, therefore, so does
  $\varphi(a')$. This shows that
  $(\varphi\times\varphi)^{-1}\sigma_L^B\subseteq\sigma_{\varphi^{-1}(L)}^A$.

  Conversely, suppose that $(a,a')\in \sigma_{\varphi^{-1}(L)}^A$ and
  $t(x_1,\ldots,x_m)$ is a term linear in $x_1$ such that
  $t_B(\varphi(a),b_2,\ldots,b_m)\in L$ with the $b_j$ in $B$. For
  each $j\in\{2,\ldots,m\}$, let $a_j\in A$ be such that
  $b_j=\varphi(a_j)$. Then, we must have
  $t_A(a,a_2,\ldots,a_m)\in\varphi^{-1}(L)$, which yields
  $t_A(a',a_2,\ldots,a_m)\in\varphi^{-1}(L)$, which in turn entails
  $t_B(\varphi(a'),b_2,\ldots,b_m)\in L$. This establishes the reverse
  inclusion
  $\sigma_{\varphi^{-1}(L)}^A\subseteq(\varphi\times\varphi)^{-1}\sigma_L^B$.
\end{proof}

\section{Profiniteness}
\label{sec:profiniteness}

Since ``Stone algebra'' already has a different meaning in the
literature, we call a \emph{Stone topological algebra} a topological
algebra whose underlying topological space is a Stone space.

To give an application of Proposition~\ref{p:lift-synt-congr}, we
first need some connections of syntactic congruences with
profiniteness. Note that, as a result of the characterization of
syntactic congruences in Lemma~\ref{l:synt-congr}, the syntactic
congruence of a clopen subset of a topological algebra is always
closed.

\begin{Thm}
  \label{t:profinite-synt-congr}
  A Stone topological algebra $A$ is profinite if and only if, for
  every clopen subset $L$ of~$A$, the syntactic congruence
  $\sigma_L^A$ is clopen.
\end{Thm}

\begin{proof}
  ($\Rightarrow$) It follows from residual finiteness and compactness
  that there is a continuous homomorphism $\varphi:A\to B$ onto a
  finite algebra $B$ such that $L=\varphi^{-1}(\varphi(L))$
  \cite[Lemma~4.1]{Almeida:2002a}.
  The kernel\footnote{By the \emph{kernel} of a mapping $\varphi:S\to
    T$ we mean the equivalence relation on~$S$ given by %
    $\{(s_1,s_2)\in S^2:\varphi(s_1)=\varphi(s_2)\}$.} of $\varphi$ is
  thus a clopen congruence on~$A$ that saturates~$L$. By
  Lemma~\ref{l:synt-congr}, $\sigma_L^A$ contains the kernel
  of~$\varphi$. Hence, the classes of $\sigma_L^A$ are also clopen, as
  they are unions of classes of the kernel of~$\varphi$.

  ($\Leftarrow$) We need to show that $A$ is residually finite. Given
  distinct points $a,a'\in A$, since $A$ is a Stone space, there is a
  clopen subset $L\subseteq A$ such that $a\in L$ and $a'\notin L$.
  Since $\sigma_L^A$ is clopen, the canonical homomorphism $A\to
  A/\sigma_L^A$ is a continuous homomorphism onto a finite (discrete)
  algebra that separates the points $a$ and $a'$, as $\sigma_L^A$
  saturates $L$ by Lemma~\ref{l:synt-congr}.
\end{proof}

We are now ready to prove the following result. The special case where
the signature is discrete was first proved
in~\cite[Theorem~4.3]{Gehrke:2016a} using duality theory.

\begin{Thm}
  \label{t:quotient-profinite}
  Let $A$ be a profinite algebra and let $\theta$ be a closed
  congruence on~$A$ such that $A/\theta$ is 0-dimensional. Then the
  quotient $A/\theta$ is profinite.
\end{Thm}

\begin{proof}
  Let $B=A/\theta$ and let $\varphi:A\to B$ be the canonical
  homomorphism. By assumption and
  Proposition~\ref{p:quotient-Hausdorff}, $B$ is a Stone topological
  algebra. We apply the criterion for $B$ to be profinite of
  Theorem~\ref{t:profinite-synt-congr}. So, let $L$ be a clopen subset
  of~$B$. By Proposition~\ref{p:lift-synt-congr}, we obtain the
  equality
  \begin{displaymath}
    (\varphi\times\varphi)^{-1}\sigma_L^B=\sigma_{\varphi^{-1}(L)}^A.
  \end{displaymath}
  Since $\varphi$ is continuous, the set $\varphi^{-1}(L)$ is clopen.
  Hence, by Theorem~\ref{t:profinite-synt-congr},
  $\sigma_{\varphi^{-1}(L)}^A$ is a clopen congruence, and so
  $A/\sigma_{\varphi^{-1}(L)}^A$ is finite. By
  Proposition~\ref{p:lift-synt-congr}, the algebra $B/\sigma_L^B$ is
  finite. Since $\sigma_L^B$ is a closed congruence by
  Corollary~\ref{c:synt-congr-closed}, it follows that $\sigma_L^B$ is
  clopen.
\end{proof}

A compact space $X$ has a natural uniform structure in the sense of
\cite[Chapter~II]{Bourbaki:1998GTa} with \emph{base} consisting of the
open neighborhoods of the diagonal $\{(x,x):x\in X\}$, that is, its
\emph{entourages} are the subsets of $X\times X$ that contain such
neighborhoods. This is the only uniform structure that determines the
topology of~$X$. In case $X$ is a Stone space, one may take as basic
members the sets of the form $\bigcup_{i=1}^nL_i\times L_i$ where the
$L_i$ constitute a clopen partition of~$X$; we call the union
$\bigcup_{i=1}^nL_i\times L_i$ the \emph{entourage determined} by the
partition $L_1,\ldots,L_n$. Suppose $A$ is a profinite algebra and
$L_1,\ldots,L_n$ is a clopen partition of~$A$. By
Theorem~\ref{t:profinite-synt-congr}, the set
$\theta=\bigcap_{i=1}^n\sigma_{L_i}^A$ is a clopen congruence on~$A$
that saturates each set $L_i$. Hence, to obtain a subbase of the
uniform structure of~$A$ one may consider only the clopen partitions
defined by clopen congruences.

Let $X$ be a uniform space and let $\Cl F\subseteq\Cl C(X)$. We say
that \Cl F is \emph{equicontinuous} if for every entourage $\alpha$ of
$X$ and every $x\in X$, there is an open subset $U\subseteq X$ such
that $x\in U$ and, for every $f\in\Cl F$, $f(U)\times
f(U)\subseteq\alpha$. We say that \Cl F is \emph{uniformly
  equicontinuous} if for every entourage $\alpha$ of $X$, there is an
entourage $\beta$ of~$X$ such that, for every $f\in\Cl F$, $(f\times
f)(\beta)\subseteq\alpha$. In case $X$ is compact, the two properties
are equivalent \cite[Chapter~X, \S2.1, Corollary~2]{Bourbaki:1998GTb}.
  
Here is another characterization of profiniteness for Stone
topological algebras. It is taken from
\cite[Theorem~4.4]{Schneider&Zumbragel:2017}. The proof presented here
is basically the same as that in~\cite{Schneider&Zumbragel:2017}
although it is slightly simplified thanks to the usage of syntactic
congruences.

\begin{Thm}
  \label{t:profinite-equicontinuous}
  A Stone topological algebra $A$ is profinite if and only if
  $M(A)$ is equicontinuous.
\end{Thm}

\begin{proof}
  ($\Rightarrow$) Let $L_1,\ldots,L_n$ be a clopen partition of~$A$
  and let $\theta=\bigcap_{i=1}^n\sigma_{L_i}^A$. Then, by
  Lemma~\ref{l:synt-congr}, for every $f\in M(A)$, we have
  \begin{displaymath}
    (f\times f)(\theta)\subseteq\bigcup_{i=1}^nL_i\times L_i.
  \end{displaymath}

  ($\Leftarrow$) We apply the criterion of
  Theorem~\ref{t:profinite-synt-congr}. So, let $L$ be a clopen subset
  of~$A$ and consider the entourage $\alpha_L$ determined by the clopen
  partition \hbox{$L,A\setminus L$}. By (uniform) equicontinuity, there
  exists a clopen partition $L_1,\ldots,L_n$ of~$A$, determining an
  entourage $\beta$ of~$X$, such that, for every $f\in M(A)$, the
  inclusion $(f\times f)(\beta)\subseteq\alpha_L$ holds or,
  equivalently, $\beta\subseteq\bigcap_{f\in M(A)}(f\times
  f)^{-1}(\alpha_L)$. By Remark~\ref{rk:synt-congr}, we have
  \begin{displaymath}
    \bigcap_{f\in M(A)}(f\times f)^{-1}(\alpha_L)=\sigma_L^A.
  \end{displaymath}
  Hence, we have $\beta\subseteq\sigma_L^A$. It follows that each
  class of~$\sigma_L^A$ is a union of some of the $L_i$ and, therefore
  it is clopen. Hence, $\sigma_L^A$ is a clopen congruence and $A$ is
  profinite by Theorem~\ref{t:profinite-synt-congr}.
\end{proof}
  
A subset of a topological space $X$ is said to be \emph{relatively
  compact} if it is contained in a compact subset of~$X$. The
following is a reformulation of
Theorem~\ref{t:profinite-equicontinuous}. As observed
in~\cite{Schneider&Zumbragel:2017}, the equivalence between the
criteria of Theorems~\ref{t:profinite-equicontinuous}
and~\ref{t:profinite-rel-compact} is an immediate application of the
Arzelà-Ascoli theorem of functional analysis \cite[Chapter~X, §2.5,
Corollary~3]{Bourbaki:1998GTb}.

\begin{Thm}[\cite{Schneider&Zumbragel:2017}]
  \label{t:profinite-rel-compact}
  A Stone topological algebra $A$ is profinite if and only if
  $M(A)$ is relatively compact in~$\Cl C(A)$.
\end{Thm}

\section{Determination of syntactic congruences}
\label{sec:cdsc}

In this section, we examine several ways of describing syntactic
congruences.

\subsection{Determination by terms and functions}
\label{sec:determ-terms-funct}

We say that a syntactic congruence $\sigma_L$ of an algebra $A$ is
\emph{finitely determined by terms} if there exists a finite subset
$F$ of $T_\Omega(\{x_1,\ldots,x_m\})$ such that $\sigma_L$ consists of
all pairs $(a,a')\in A^2$ for which
\begin{equation*}
  \forall t\in F\ \forall b_2,\ldots,b_m\in A\ 
  \bigl(
  t_A(a,b_2,\ldots,b_m)\in L \iff t_A(a',b_2,\ldots,b_m)\in L
  \bigr).
\end{equation*}
We then also say that the set of terms $F$ \emph{determines}
$\sigma_L$. In Lemma~\ref{l:from-terms-to-linear-terms} below, it is
observed that if $\sigma_L$ is determined by a finite set of terms,
then it is also determined by such a set in which every term is linear
in the first variable. For instance, semigroups, monoids, rings, and
distributive lattices all have syntactic congruences determined by
finite sets of terms for all subsets, even in a uniform way in the
sense that the same finite set of terms works for all syntactic
congruences on all algebras of the chosen type.

Generalizing a result of Numakura~\cite{Numakura:1957} (for the cases
of semigroups and distributive lattices), the first author
\cite{Almeida:1989b} has shown that a sufficient condition for a Stone
topological algebra to be profinite is that its syntactic congruences
of clopen subsets are finitely determined (see also
\cite{Almeida&Weil:1994,Clark&Davey&Freese&Jackson:2004}).

There is a stronger form of determination of a syntactic congruence,
which is suggested by Lemma~\ref{l:synt-congr}. We say that the
syntactic congruence $\sigma_L$ of a subset $L$ of an algebra $A$ is
\emph{$S$-determined} if $S$ is a set of functions $A\to A$ such that,
for all $a,a'\in A$, $(a,a')\in\sigma_L$ holds if and only if
\begin{displaymath}
  \forall f\in S\ \bigl(f(a)\in L \iff f(a')\in L\bigr).
\end{displaymath}
For instance, $\sigma_L$ is $M(A)$-determined by
Lemma~\ref{l:synt-congr}. As in Remark~\ref{rk:synt-congr}, note that
$\sigma_L$ is $S$-determined if and only if
\begin{displaymath}
  \sigma_L = \bigcap_{f\in
    S}(f\times f)^{-1}(\alpha_L) = \bigcap_{f\in S}\alpha_{f^{-1}(L)}.
\end{displaymath}

\begin{Prop}
  \label{p:finite-determined}
  Let $A$ be an algebra and $L$ a subset of $A$. If the syntactic
  congruence $\sigma_L$ has finite index then it is $F$-determined for
  some finite subset $F$ of~$M(A)$.
\end{Prop}

\begin{proof}
  Consider the syntactic homomorphism $\eta:A\to A/\sigma_L$, which
  sends each $a\in A$ to its syntactic class $a/\sigma_L$. By
  assumption, the algebra $A/\sigma_L$ is finite. Hence,
  $M(A/\sigma_L)$ is a finite monoid. For each $f\in M(A/\sigma_L)$,
  we may choose a term %
  $t\in T_\Omega(\{x_1,\ldots,x_{k+1}\})$ which is linear in $x_1$ and
  elements $a_2,\ldots,a_{k+1}\in A$ such that
  $f(x/\sigma_L)=t_A(x,a_2,\ldots,a_{k+1})/\sigma_L$ for every $x\in
  A$. We let $\hat f(x)=t_A(x,a_2,\ldots,a_{k+1})$ for each $x\in A$,
  which defines an element $\hat f$ of~$M(A)$ such that $\eta\circ\hat
  f=f\circ\eta$. Let $F=\{\hat f:f\in M(A/\sigma_L)\}$. We claim that
  $\sigma_L$ is $F$-determined.

  Since $F\subseteq M(A)$, we have %
  $\sigma_L\subseteq\bigcap_{g\in F}(g\times g)^{-1}(\alpha_L)$.
  Conversely, suppose $a,a'\in A$ are such that, for every $f\in
  M(A/\sigma_L)$, we have $\hat f(a)\in L$ if and only if $\hat
  f(a')\in L$. Since $\eta$ saturates $L$, $x\in L$ is equivalent to
  $\eta(x)\in\eta(L)$. Thus, we get $f(\eta(a))\in\eta(L)$ if and only
  if $f(\eta(a'))\in\eta(L)$ for every $f\in M(A/\sigma_L)$. By
  definition of the syntactic congruence, we deduce that
  $(\eta(a),\eta(a'))\in\sigma_{\eta(L)}$, that is,
  $(a,a')\in(\eta\times\eta)^{-1}(\sigma_{\eta(L)})$. By
  Proposition~\ref{p:lift-synt-congr}, it follows that
  $(a,a')\in\sigma_L$, which establishes the claim.
\end{proof}

Note that Proposition~\ref{p:finite-determined} applies in particular
to clopen syntactic congruences $\sigma_L$ on a compact algebra $A$,
which are necessarily determined by a clopen subset $L$ of~$A$. One
may ask whether the same finite subset $F$ of $M(A)$ may be used to
determine all syntactic congruences of clopen subsets $L$ of~$A$. This
is trivially the case for finite algebras but here is an infinite
example for which such a finite set $F$ also exists.

\begin{eg}
  \label{eg:globally-finite-determined}
  Let $A$ be a Stone space and consider a constant binary operation
  on~$A$. This turns $A$ into a profinite semigroup. Note that, for
  every $L\subseteq A$, we have $\sigma_L=\alpha_L$. Thus, all
  syntactic congruences are $F$-determined for every subset $F$
  of~$M(A)$. The same conclusion would be reached if we took the empty
  signature instead.\qed
\end{eg}

The next result shows that such an example cannot be simultaneously
finitely generated and infinite.

\begin{Prop}
  \label{p:fg-non-globally-finite-determined}
  If $A$ is an infinite finitely generated profinite algebra then, for
  every finite subset $F$ of~$M(A)$, there exists a clopen subset $L$
  of~$A$ such that $\sigma_L$ is not $F$-determined.
\end{Prop}

\begin{proof}
  Suppose that $F$ is a finite subset of $M(A)$ that determines the
  syntactic congruence of every clopen subset $L$ of $A$, that is, its
  syntactic congruence $\sigma_L$ is given by the formula
  $\sigma_L=\bigcap_{f\in F}\alpha_{f^{-1}(L)}$. Note that each
  equivalence relation $\alpha_{f^{-1}(L)}$ has at most two classes.
  Hence, $\sigma_L$ has at most $2^{|F|}$ classes, that is, the
  syntactic algebra $A/\sigma_L$ has at most $2^{|F|}$ elements. Now,
  given a finite set $S$ and a finite generating set $X$ of~$A$, there
  are only finitely many functions $X\to S$. Thus, on subsets of $S$
  there are only finitely many algebraic structures which are
  homomorphic images of~$A$. There are only finitely many congruences
  that are kernels of these homomorphisms from $A$ to~$S$. In
  particular, there are only finitely many syntactic congruences
  $\sigma_L$ of clopen subsets $L$ of~$A$. Since the corresponding
  syntactic homomorphisms $A\to A/\sigma_L$ suffice to separate the
  points of~$A$, we conclude that $A$ is finite, in contradiction with
  the hypothesis.
\end{proof}

\subsection{Compact determination}
\label{sec:comp-determ}

Suppose that $A$ is a topological algebra and $L\subseteq A$. We say
that the syntactic congruence $\sigma_L$ is \emph{compactly
  determined} if $\sigma_L$ is $C$-determined for some compact subset
$C$ of~$\cont(A)$. Similarly, $\sigma_L$ is \emph{finitely determined}
if $\sigma_L$ is $F$-determined for some finite subset $F$
of~$\cont(A)$. This not to be confused with $\sigma_L$ being
determined by a finite set of terms, which is, a priori, a weaker
property in case $F\subseteq M(A)$.
Theorem~\ref{t:syntactic-clopen-in-compact-algebra} at the end of
Subsection~\ref{sec:comp-determ-versus-finite-determ-by-terms} shows
that all these properties are equivalent in case $A$ is compact and
$L\subseteq A$ is clopen.

In this subsection, among other results, we establish that, if all
syntactic congruences of clopen subsets of a Stone topological algebra
$A$ are compactly determined, then $A$ is profinite. Several of our
results are stated for locally compact algebras because they hold in
that more general setting.

Let us briefly introduce a convenient notation for partial evaluation.
Given $f:Y\times Z\to X$, we define $f^\sharp:Z\to X^Y$ by:
\begin{equation*}
    f^\sharp(z)(y) = f(y,z).
\end{equation*}
The following proposition regroups a few useful properties of the
compact-open topology. For proofs, we refer the reader to
\cite[Chapter~X, \S3.4, Theorem~3 and Proposition~9]{Bourbaki:1998GTb}.

\begin{Prop}
  \label{p:compact-open}
  Let $X$, $Y$ and $Z$ be topological spaces, with $Y$ locally
  compact.
  \begin{enumerate}
  \item Composition is a continuous mapping
    $\cont(Y,X)\times\cont(Z,Y)\to \cont(Z,X)$.
  \item Evaluation is a continuous mapping $Y\times \cont(Y,X)\to X$.
  \item If $f\in\cont(Y\times Z, X)$, then $f^\sharp\in\cont(Z,
    \cont(Y,X))$.
  \end{enumerate}
\end{Prop}

In particular, from Proposition~\ref{p:compact-open}(1) it follows
that, for every topological space $X$, $\cont(X)$ is a topological
monoid. In case $A$ is a topological algebra, $M(A)$ is a submonoid
of~$\cont(A)$, whence so is its closure $\overline{M(A)}$.

The following statement is an elementary observation in topology whose
proof is presented for the sake of completeness.

\begin{Prop}
  \label{p:compact-intersections}
  Let $X$ and $Y$ be two topological spaces, with $X$ locally compact,
  and let $K$ be a compact subspace of $\cont(X,Y)$.
  \begin{enumerate}
  \item The following mapping is closed:
    \begin{align*}
      \Phi_K:2^Y &\to 2^X\\
      V & \mapsto  \bigcup_{f\in K}f^{-1}(V).
    \end{align*}
  \item The following mapping is open:
    \begin{align*}
      \Psi_K:2^Y &\to 2^X\\
      U & \mapsto \bigcap_{f\in K}f^{-1}(U).
    \end{align*}
  \end{enumerate}
\end{Prop}
  
\begin{proof}
  Clearly, (1) and (2) imply each other by taking the complement:
  $\Phi_K(2^Y\setminus U)=2^X\setminus \Psi_K(U)$. To prove (1),
  consider a closed subset $V$ of~$Y$. Let $\{ x_i \}_{i\in I}$ be a
  net in $\Phi_K(V)$ converging to $x\in X$. By definition of
  $\Phi_K$, there exists a net $\{f_i\}_{i\in I}$ in $K$ such that,
  for all $i\in I$, $f_i(x_i)\in V$. Since $K$ is compact, we may
  assume by taking a subnet that the first component of $\{ (f_i,x_i)
  \}_{i\in I}$ converges, to say $f\in K$. Since $V$ is closed and
  evaluation is continuous by Proposition~\ref{p:compact-open}(2), we
  obtain
  \begin{equation*}
    f(x) = \lim_{i\in I}f_i(x_i) \in V.
  \end{equation*}
  It follows that $x\in\Phi_K(V)$, thus showing that $\Phi_K(V)$ is
  closed.
\end{proof}

A further result in topology that we require for later considerations
is the following.

\begin{Lemma}
  \label{l:diagonal}
  Let $X$ be a Hausdorff topological space. The diagonal mapping
  $\Delta:\cont(X)\to\cont(X\times X)$ defined by $\Delta(f) = f\times
  f$ is continuous with respect to the compact-open topologies.
\end{Lemma}

\begin{proof}
  By \cite[Lemma~3.4.6]{Engelking:1989}, the sets of the form
  $[K,U_1\times U_2]$, where $K$ ranges over all compact subsets of
  $X\times X$ and $U_1$ and $U_2$ range over all open subsets of $X$,
  constitute a subbase for the compact-open topology of $\cont(X\times
  X)$. Therefore, it suffices to show that for each compact subset
  $K\subseteq X\times X$ and open subsets $U_1, U_2\subseteq X$, the
  set $\Delta^{-1}[K,U_1\times U_2]$ is open in the compact-open
  topology of $\cont(X)$.

  Let $p_1, p_2$ be the natural projections of $X\times X$ on its
  first and second components, respectively. Note that, for $i=1,2$,
  we have $p_i\circ\Delta(f) = f\circ p_i$. It follows that:
  \begin{align*}
    \Delta(f)(K)\subseteq U_1\times U_2
    &\iff
      p_i(\Delta(f)(K))\subseteq U_i & \text{for $i=1,2$} \\
    &\iff
      f(p_i(K))\subseteq U_i & \text{for $i=1,2$}.
  \end{align*}
  This shows that $\Delta^{-1}[K,U_1\times U_2] =
  [p_1(K),U_1]\cap[p_2(K),U_2]$. But for $i=1,2$, the projection $p_i$
  is continuous, so $p_i(K)$ is a compact subset of~$X$ and
  $[p_i(K),U_i]$ is open in the compact-open topology of $\cont(X)$.
  Hence, so is $\Delta^{-1}[K,U_1\times U_2]$.
\end{proof}

The next proposition shows that, in order to be compactly determined,
the syntactic congruence of a clopen subset merely needs to be
determined by a relatively compact subset of $\cont(A)$.

\begin{Prop}
  \label{p:closure-determine}
  Let $A$ be a locally compact algebra and $L$ be a clopen subset of
  $A$. If $F\subseteq\cont(A)$ is such that $\sigma_L$ is
  $F$-determined, then $\sigma_L$ is also $\overline{F}$-determined.
\end{Prop}

\begin{proof}
  Clearly, we have:
  \begin{equation*}
    \bigcap_{f\in\overline{F}}(f\times f)^{-1}(\alpha_L)
    \subseteq
    \bigcap_{f\in F}(f\times f)^{-1}(\alpha_L) = \sigma_L.
  \end{equation*}
  To prove the reverse inclusion, let us fix $x\in\sigma_L$ and take
  an arbitrary $f\in\overline{F}$. Let $\{ f_i \}_{i\in I}$ be a net
  in $F$ converging to~$f$. By assumption, $(f_i\times f_i)(x)$
  belongs to $\alpha_L$ for every $i\in I$. Since $\alpha_L$ is closed
  and the evaluation and diagonal mappings are continuous,
  respectively by Proposition~\ref{p:compact-open}(2) and
  Lemma~\ref{l:diagonal}, we deduce that
  \begin{equation*}
    (f\times f)(x)
    = \lim_{i\in I}(f_i\times f_i)(x)
    \in \alpha_L.
  \end{equation*}
  This shows that $x\in(f\times f)^{-1}(\alpha_L)$, as required.
\end{proof}

Combining Proposition~\ref{p:closure-determine} with
Theorem~\ref{t:profinite-rel-compact}, we get the following result.

\begin{Cor}
  \label{c:pf-compactly-determined}
  Let $A$ be a profinite algebra and $L$ be a clopen subset of $A$.
  Then $\sigma_L$ is compactly determined.
\end{Cor}

The following result shows the significance of the notion of compactly
determined syntactic congruence.

\begin{Thm}
  \label{t:lc-syntactic-clopen}
  Let $A$ be a locally compact algebra $A$ and $L$ a subset of~$A$.
  Then $L$ is clopen and $\sigma_L$ is compactly determined if and
  only if $\sigma_L$ is clopen.
\end{Thm}

\begin{proof}
  Suppose first that $L$ is clopen and $\sigma_L$ is $K$-determined,
  where $K$ is a compact subset of~$\cont(A)$. Since $\sigma_L$ is
  $K$-determined, we may write $\sigma_L = \Psi_{\Delta K}(\alpha_L)$
  in the notation of Proposition~\ref{p:compact-intersections} and
  Lemma~\ref{l:diagonal}. By Lemma~\ref{l:diagonal}, $\Delta$ is
  continuous with respect to the compact-open topologies, so $\Delta
  K$ is a compact subset of $\cont(A\times A)$. By Proposition
  \ref{p:compact-intersections} applied to $Y=X\times X$, it follows
  that $\Psi_{\Delta K}$ maps open relations on $X\times X$ to open
  relations on~$X$. But since $L$ is clopen, so is $\alpha_L$, whence
  $\sigma_L$ is clopen.

  For the converse, assume that $\sigma_L$ is clopen. If $\sigma_L$ is
  the universal relation, then $L$ is either $\emptyset$ or $A$ and
  $\sigma_L$ is $K$-determined for every subset
  $K$ of~$\cont(A)$. Otherwise, we may choose elements
  $a,b\in A$ that are not $\sigma_L$-equivalent. Let $K$
  be the set of all mappings $A\to\{a,b\}$ that are constant on each
  $\sigma_L$-class. Since $\sigma_L$ is clopen, $K$ is
  contained in $\cont(A)$. Consider the mapping
  \begin{align*}
    \varphi:\cont(A/\sigma_L,\{a,b\})&\to K\\
    f&\mapsto f\circ\eta,
  \end{align*}
  where $\eta:A\to A/\sigma_L$ is the natural quotient mapping. Note
  that $A/\sigma_L$ is a discrete space under the quotient topology,
  whence it is locally compact. By
  Proposition~\ref{p:compact-open}(1), $\varphi$ is a continuous
  mapping. As both spaces $A/\sigma_L$ and $\{a,b\}$ are discrete, the
  space $\cont(A/\sigma_L,\{a,b\})$ is in fact the product space
  $\{a,b\}^{A/\sigma_L}$. Since $\varphi$ is onto and continuous, we
  deduce that $K$ is compact. The proof is achieved by observing that
  $\sigma_L$ is $K$-determined.
\end{proof}

Note that the second part of the above proof does not use the
hypothesis that the algebra $A$ is locally compact. In the locally
compact case, the proof could also be given by invoking the
Arzel\`a-Ascoli theorem.

Theorems~\ref{t:lc-syntactic-clopen} and~\ref{t:profinite-synt-congr}
yield the following result.

\begin{Cor}
  \label{c:ST-cd->profinite}
  Let $A$ be a Stone topological algebra, and suppose that for every
  clopen subset $L$ of $A$, $\sigma_L$ is compactly determined. Then
  $A$ is profinite.
\end{Cor}

Given a profinite algebra $A$ and a clopen subset $L$ of $A$, one may
wonder what are the compact subsets of $\cont(A)$ determining
$\sigma_L$. The following shows that, at the very least, there is
always one that is minimal.

\begin{Prop}
  \label{p:lc-minimal}
  Let $A$ be locally compact algebra, $L$ be a clopen subset of~$A$.
  Then, every compact subset of $\cont(A)$ determining $\sigma_L$
  contains a minimal compact subset determining $\sigma_L$.
\end{Prop}

\begin{proof}
  We apply Zorn's lemma. Let $F$ be a compact subset of
  $\cont(A)$ determining $\sigma_L$. Fix a descending chain
  $\{F_i\}_{i\in I}$ of closed subsets of $F$
  that determines $\sigma_L$ and let:
  \begin{equation*}
    F' = \bigcap_{i\in I}F_i.
  \end{equation*}
  We want to show that $F'$ determines $\sigma_L$. Fixing an arbitrary
  $i\in I$, the inclusion $F'\subseteq F_i$ gives:
  \begin{equation}
    \label{eq:lc-minimal-1}
    \bigcap_{f\in F'}(f\times f)^{-1}(\alpha_L)
    \supseteq
    \bigcap_{f\in F_i}(f\times f)^{-1}(\alpha_L) = \sigma_L.
  \end{equation}
  It remains to show the reverse inclusion. Let us suppose that $x\in
  A^2\setminus\sigma_L$. Then, we have
  \begin{equation*}
    \forall i\in I\ \exists f_i\in F_i,\
    (f_i\times f_i)(x)\notin\alpha_L.
  \end{equation*}
  This defines a net $\{ f_i \}_{i\in I}$ in $F$, which by compactness
  has a converging subnet $\{ f_{i_j} \}_{j\in J}$; let $f'\in F$
  denote its limit.
  For each $i\in I$, there is $k\in J$ such that $i_k\geq i$, and so
  the net $\{f_{i_j}\}_{j\in J}$ is eventually in $F_i$. Since $F_i$
  is closed, it follows that $f'\in F_i$, and this holds for each
  $i\in I$.
  Thus, $f'$ belongs to~$F'$. Furthermore, as the evaluation and
  diagonal mappings are continuous, again respectively by
  Proposition~\ref{p:compact-open}(2) and Lemma~\ref{l:diagonal}, and
  $\alpha_L$ is clopen, we get
  \begin{equation*}
    (f'\times f')(x) = \lim_{j\in J}(f_{i_j}\times f_{i_j})(x) \notin\alpha_L.
  \end{equation*}
  This shows that:
  \begin{equation*}
    x\in
    \bigcup_{f\in F'}(f\times f)^{-1}(A^2\setminus\alpha_L)
    = A^2\setminus
    \bigcap_{f\in F'}(f\times f)^{-1}(\alpha_L).
  \end{equation*}
  Thus, the reverse inclusion in~\eqref{eq:lc-minimal-1} is proved and
  $\sigma_L$ is $F'$-determined.
\end{proof}

In case $A$ is a compact algebra and $L$ is a clopen subset of~$A$,
one may use the ideas in the proof of
Proposition~\ref{p:finite-determined} to show that every minimal
compact subset of $\cont(A)$ determining $\sigma_L$ is finite.
Combining with Proposition~\ref{p:lc-minimal}, it follows that every
compact subset of~$\cont(A)$ determining $\sigma_L$ contains a finite
such set. The following example shows that
Proposition~\ref{p:lc-minimal} cannot be improved in the same
direction for arbitrary locally compact algebras.

\begin{eg}
  \label{eg:cd-not-fd}
  Consider the additive monoid $\mathbb{N}$ of natural numbers under
  the discrete topology. Let $L$ be an infinite subset of~$\mathbb{N}$
  containing no infinite arithmetic progression, for instance the set
  of all powers of 2 or the set of all primes. We claim that:
  \begin{enumerate}
  \item\label{item:cd-not-fd-1} $\sigma_L$ is the equality relation;
  \item\label{item:cd-not-fd-2} $\sigma_L$ is not finitely determined.
  \end{enumerate}
  As $\mathbb{N}$ is locally compact, being discrete, and $\sigma_L$
  is clopen, for the same reason, Theorem~\ref{t:lc-syntactic-clopen}
  yields that $\sigma_L$ is compactly determined. Hence,
  by~(\ref{item:cd-not-fd-2}), $\sigma_L$ is an example of a compactly
  determined syntactic congruence of a clopen subset of a locally
  compact algebra that is not finitely determined.

  To prove Claim~(\ref{item:cd-not-fd-1}), since $\mathbb{N}$ is a
  commutative monoid, a pair $(m,n)$ of natural numbers belongs to
  $\sigma_L$ if and only if
  \begin{equation}
    \label{eq:cd-not-fd-1}
    \forall x\in\mathbb{N}\ (m+x\in L \iff n+x\in L).
  \end{equation}
  Suppose that Property~\eqref{eq:cd-not-fd-1} holds with $m<n$. Since
  $L$ is infinite, there exists $k$ such that $m+k$ belongs to~$L$ and
  whence so does $n+k$. Using~\eqref{eq:cd-not-fd-1}, we deduce that
  all the elements in the arithmetic progression starting with $m+k$
  with period $n-m$ lie in~$L$, which contradicts the assumption on
  the set $L$. Hence no two distinct elements of~$\mathbb{N}$ can be
  $\sigma_L$-equivalent, thereby proving (\ref{item:cd-not-fd-1}).

  To establish Claim~(\ref{item:cd-not-fd-2}), it suffices to observe
  that the argument at the beginning of the proof of
  Proposition~\ref{p:fg-non-globally-finite-determined} shows that a
  finitely determined syntactic congruence has finite index, which is
  not the case of~$\sigma_L$ by~(\ref{item:cd-not-fd-1}).\qed
\end{eg}

\subsection{Compact determination versus finite determination by
  terms}
\label{sec:comp-determ-versus-finite-determ-by-terms}

We next show that, for compact algebras, a syntactic congruence is
compactly determined if and only if it is determined by a finite set
of terms, in the sense introduced at the beginning of
Subsection~\ref{sec:determ-terms-funct}.

We start with some notation. Let $A$ be an algebra, $F$ be a subset of
$\cont(A)$, and $S$ be a subset of $\cont(A^{k+1},A)$
Given $s\in S$,
in view of Proposition~\ref{p:compact-open}, we obtain a function
$s^\sharp\in\cont(A^k,\cont(A))$. We define:
\begin{equation*}
  F_S
  = \{ f\in F: \exists s\in S\ \exists v\in A^k, f = s^\sharp(v) \}.
\end{equation*}
We also abbreviate $F_{\{s\}}$ by $F_{s}$. Note that
\begin{equation*}
  F_S
  =\bigcup_{s\in S} F_s
  =F\cap\left(\bigcup_{s\in S}s^\sharp(A^k)\right)
  \subseteq \cont(A).
\end{equation*}
In case $T$ is subset of~$T_\Omega(\{x_1,\dots,x_{k+1}\})$, we also
write $F_T$ for $F_S$, where $S=\{t_A: t\in T\}$.

\begin{Lemma}
  \label{l:from-terms-to-linear-terms}
  Let $L$ be a subset of an algebra $A$. Then $\sigma_L$ is determined
  by a finite set of terms if and only if it is determined by some set
  of the form $F_T$, where $T$ is a finite subset of
  $T_\Omega(\{x_1,\dots,x_{k+1}\})$ for some $k\ge0$, consisting of
  terms linear in $x_1$, and $F\subseteq M(A)$.
\end{Lemma}

\begin{proof}
  The if part of the statement of the lemma is trivial. For the
  converse, we first associate with each term $t\in
  T_\Omega(\{x_1,\ldots x_{k+1}\})$ a term $s\in T_\Omega(X)$ with
  $X=\{y_1,\ldots,y_r,x_2,\ldots,x_{k+1}\}$ by replacing each
  occurrence of $x_1$ by a distinct $y_i$. Let
  \begin{displaymath}
    s_i=s_{T_\Omega(X)}(y^{(i-1)},x,z^{(r-i)},x_2\ldots,x_{k+1})
    \in T_\Omega(x,y,z,x_2,\ldots,x_{k+1}),
  \end{displaymath}
  where $u^{(\ell)}$ stands for $\ell$ components equal to $u$. Then
  each term $s_i$ is linear in $x$ and the following formulas hold for
  $a,a'\in A$ and $v\in A^k$:
  \begin{align*}
    (s_1)_A^\sharp(a,a',v)(a')&=t_A^\sharp(v)(a')\\
    (s_r)_A^\sharp(a,a',v)(a)&=t_A^\sharp(v)(a)\\
    (s_i)_A^\sharp(a,a',v)(a)&=(s_{i+1})_A^\sharp(a,a',v)(a')
    \quad(i=1,\ldots,r-1).
  \end{align*}
  It follows that
  \begin{align*}
    \sigma_L
    &\subseteq
    \bigcap_{i=1,\ldots,r;\ b,b'\in A}
    \bigl((s_i)_A^\sharp(b,b',v)\times(s_i)_A^\sharp(b,b',v)\bigr)^{-1}(\alpha_L) \\
    &\subseteq
    \bigl(t_A^\sharp(v)\times t_A^\sharp(v)\bigr)^{-1}(\alpha_L),
  \end{align*}
  which shows that $\sigma_L$ is also determined by a finite set of
  terms that are linear in~$x_1$ by simply taking, for a finite set of
  terms determining $\sigma_L$, the union of the sets of terms
  constructed above for each term in the given set.
\end{proof}

The next lemma examines how the operation $F\mapsto F_T$ behaves with
respect to topological closure when $T$ is a finite set of terms.

\begin{Lemma}
  \label{l:closure}
  Let $A$ be a compact algebra, $F$ be a subset of $\cont(A)$ and $T$
  be a finite set of terms in $T_\Omega(\{x_1,\dots,x_{k+1}\})$. Then,
  $\overline{F_T}$ is contained in~$\overline{F}_T$.
\end{Lemma}

\begin{proof}
  Let $f\in\overline{F_T}$. Then, we can write $f$ as a
  limit
  \begin{equation*}
    f = \lim_{i\in I}f_i,
  \end{equation*}
  where $\{f_i\}_{i\in I}$ is a net in $F_T$. For each $i\in I$,
  choose $t_i\in T$ and $v_i\in A^k$ such that $f_i =
  (t_i)_A^\sharp(v_i)$. Since $T$ is finite and $A^k$ is compact, we
  may extract a subnet $\{f_{i_j}\}_{j\in J}$ such that
  $\{t_{i_j}\}_{j\in J}$ takes a constant value $t\in T$ and
  $\{v_{i_j}\}_{j\in J}$ converges to $v$ in~$A^k$. By
  Proposition~\ref{p:compact-open}(3), $t_A^\sharp$ is continuous and
  it follows that
  \begin{equation*}
    f = \lim_{j\in J}f_{i_j}
    = \lim_{j\in J} t_A^\sharp(v_{i_j})
    = t_A^\sharp(v).
  \end{equation*}
  Hence, $f$ lies in $\overline{F}_T$.
\end{proof}

We are now ready to achieve the goal announced at the beginning of
this subsection.

\begin{Prop}
  \label{p:lc-compactly-determined}
  Let $A$ be a 
  compact algebra, and $L$ be a clopen subset of $A$. If $\sigma_L$ is
  determined by a finite set of terms, then it is compactly
  determined.
\end{Prop}

\begin{proof}
  By Lemma~\ref{l:from-terms-to-linear-terms}, there is a finite
  subset $T$ of~$T_\Omega(\{x_1,\dots,x_{k+1}\})$ consisting of terms
  linear in $x_1$ and a subset $F$ of $M(A)$ such that
  $F_T$ determines $\sigma_L$. Then, by Proposition
  \ref{p:closure-determine}, $\sigma_L$ is also determined by
  $\overline{F_T}$. By
  Lemma~\ref{l:closure}, we have
  \begin{equation*}
    \overline{F_T}
    \subseteq
    \overline{F}_T
    \subseteq
    M(A).
  \end{equation*}
  Since both $\overline{F_T}$ and $M(A)$ determine $\sigma_L$, so does
  $\overline{F}_T$.
  We claim that $\overline{F}_T$ is compact. Indeed, for
  each $t\in T$, we have
  \begin{equation*}
    \overline{F}_t
    = \overline{F}\cap t_A^\sharp(A^k).
  \end{equation*}
  But note that $t_A^\sharp(A^k)$ is compact, because $A^k$ is compact
  and $t_A^\sharp$ is continuous by
  Proposition~\ref{p:compact-open}(3). Therefore, $\overline{F}_t$,
  being an intersection of a closed with a compact subset, is itself
  compact. It follows that $\overline{F}_T =\bigcup_{t\in
    T}\overline{F}_t$ is compact, as claimed.
\end{proof}

Combining Theorem~\ref{t:lc-syntactic-clopen} and
Propositions~\ref{p:finite-determined}
and~\ref{p:lc-compactly-determined}, we obtain the following main
result of this section.

\begin{Thm}
  \label{t:syntactic-clopen-in-compact-algebra}
  The following conditions are equivalent for a compact algebra $A$
  and a subset $L\subseteq A$:
  \begin{enumerate}
  \item $\sigma_L$ is clopen;
  \item $L$ is clopen and $\sigma_L$ is compactly determined;
  \item $L$ is clopen and $\sigma_L$ is finitely determined;
  \item $L$ is clopen and $\sigma_L$ is finitely determined by a set
    of terms;
  \item the quotient algebra $A/\sigma_L$ is finite and discrete.
  \end{enumerate}
\end{Thm}

The following example shows that it is not possible to extend
Proposition~\ref{p:lc-compactly-determined} for the case of locally
compact algebras. In fact, we exhibit a locally compact semigroup and
a clopen subset whose syntactic congruence is not clopen. This
syntactic congruence is not compactly determined by
Theorem~\ref{t:lc-syntactic-clopen}, while it is determined by a
finite set of terms as in fact every syntactic congruence of a
semigroup has this property.
 					
\begin{eg}
  \label{eg:counterexample-converse-lc-compactly-determined}
  We consider the topological semigroup obtained as the direct product
  of the following locally compact semigroups $A$ and $B$, so that it
  is locally compact. Let $A=\mathbb N$ be the discrete semigroup of
  natural numbers with maximum as operation, and let $B$ be the one
  point compactification of the usual additive semigroup of natural
  numbers $\mathbb N$. While the description of the locally compact
  semigroup $A$ is clear, we describe the compact semigroup $B$ in
  more detail: we have $B=\mathbb N\cup \{\infty\}$, where $\infty$ is
  the new point for which we put $\infty + n=n + \infty=\infty
  +\infty=\infty$, for $n\in \mathbb N$. We also recall that the open
  sets in $B$ are all subsets of $\mathbb N$ together with the subsets
  of the form $\{\infty\}\cup(\mathbb N\setminus F)$, where $F$ is a
  finite subset of $\mathbb N$.

  Now we take $L=\{(n,n) \mid n\in \mathbb N\}$, which is clopen in
  $A\times B$ because
  \begin{displaymath}
    L=\bigcup_{n\in \mathbb N} \{(n,n)\}
    \quad \text{and} \quad 
    (A\times B ) \setminus L
    =\bigcup_{n\in \mathbb N} \{n\}\times (B\setminus \{n\}).
  \end{displaymath}
  We claim that one class of the syntactic congruence $\sigma_L$ is
  $A\times \{\infty\}$. All elements in $A\times \{\infty\}$ are
  $\sigma_L$-related as it is not possible to multiply them by any
  element and obtain a result in $L$. To show that elements from
  $A\times \{\infty\}$ are not $\sigma_L$-related with other elements,
  consider pairs $(i,j)$ and $(k,\infty)$ with $i,j,k\in \mathbb N$.
  Then $(i+j,i) \cdot (i,j)=(i+j,i+j)\in L$ but $(i+j,i) \cdot
  (k,\infty)=(\max\{i+j,k\},\infty)\not \in L$. Hence, $(i,j)$ and
  $(k,\infty)$ are not $\sigma_L$-equivalent. Finally, we claim that
  $A\times \{\infty\}$ is not open, which establishes that $\sigma_L$
  is not clopen. To show that $A\times \{\infty\}$ is not open, recall
  first that in the product space $A\times B$, a base of the topology
  consists of open subsets $O \times O'$ with $O$ and $O'$ open
  subsets respectively of $A$ and $B$. However, if we consider
  $(k,\infty)$ in such $O \times O'$, then there is also some element
  $(k,\ell)$ in $O \times O'$ with the same first coordinate and $\ell
  \in\mathbb N$.

  This example has %
  another feature that it is worth noting.
  The algebra $A\times B$ is residually finite. Hence, while the
  condition that the syntactic congruence of a clopen subset of a
  locally compact 0-dimensional algebra is always clopen implies that
  the algebra is residually discrete, the converse fails even under
  the stronger assumption of residual finiteness. This is in contrast
  with the case of compact algebras, for which the two conditions are
  equivalent by Theorem~\ref{t:profinite-synt-congr}.\qed
\end{eg}

\section{Summary of results and conclusion}
\label{sec:summary}

In conclusion, we have the following result building on the various
characterizations of profiniteness in Stone topological algebras
presented in this paper.

\begin{Thm}
  \label{t:summary}
  The following conditions are equivalent for a Stone topological
  algebra $A$:
  \begin{enumerate}
  \item\label{item:summary-1} $A$ is profinite;
  \item\label{item:summary-5} for every clopen subset $L\subseteq A$,
    $\sigma_L$ is a clopen congruence;
  \item\label{item:summary-2} $M(A)$ is equicontinuous;
  \item\label{item:summary-3} $M(A)$ is relatively compact in~$\Cl
    C(A)$;
  \item\label{item:summary-4} the closure of $M(A)$ in~$\Cl C(A)$
    is a profinite submonoid;
  \item\label{item:summary-6} for every clopen subset $L\subseteq
    A$, there exists a continuous homomorphism $\varphi:A\to B$
    onto a finite algebra $B$ such that
    $L=\varphi^{-1}(\varphi(L))$;
  \item\label{item:summary-7} for every clopen subset $L\subseteq A$,
    the congruence $\sigma_L$ is determined by some finite set of
    terms;
  \item\label{item:summary-8} for every clopen subset $L\subseteq A$,
    the congruence $\sigma_L$ is $F$-determined by some finite subset
    $F$ of $M(A)$;
  \item\label{item:summary-9} for every clopen subset $L\subseteq A$,
    the congruence $\sigma_L$ is $F$-determined by some finite subset
    $F$ of $\Cl C(A)$;
  \item\label{item:summary-10} for every clopen subset $L\subseteq A$,
    the congruence $\sigma_L$ is $C$-determined by some compact subset
    $C$ of $\Cl C(A)$.
  \end{enumerate}
\end{Thm}

\begin{proof}
  By Theorems~\ref{t:profinite-synt-congr},
  \ref{t:profinite-equicontinuous} and~\ref{t:profinite-rel-compact},
  we have the equivalences
  \eqref{item:summary-5}$\,\Leftrightarrow\,$\eqref{item:summary-1}
  $\,\Leftrightarrow\,$\eqref{item:summary-2}
  $\,\Leftrightarrow\,$\eqref{item:summary-3}. In fact, in the proof
  of Theorem~\ref{t:profinite-synt-congr}, we showed that
  \eqref{item:summary-1}$\,\Rightarrow\,$\eqref{item:summary-6}$\,\Rightarrow\,$\eqref{item:summary-5},
  whence we also have
  \eqref{item:summary-1}$\,\Leftrightarrow\,$\eqref{item:summary-6}

  The implication
  \eqref{item:summary-4}$\,\Rightarrow\,$\eqref{item:summary-3} is
  obvious. To establish the reverse implication
  \eqref{item:summary-3}$\,\Rightarrow\,$\eqref{item:summary-4} it is
  enough to prove that $\overline{M(A)}$ is a 0-dimensional
  topological monoid, which holds if so is $\cont(A)$. Now, given a
  clopen subset $L\subseteq A$, we have
  \begin{displaymath}
    \Cl C(A)\setminus [K,L]=\bigcup_{a\in K}[a,A\setminus L],
  \end{displaymath}
  which shows that $[K,L]$ is closed for every subset $K$ of~$A$.
  Thus, if $K,L\subseteq A$ are clopen, then so is $[K,L]$. Hence,
  $\Cl C(A)$ is 0-dimensional and it was already observed after
  Proposition~\ref{p:compact-open} that $\cont(A)$ is a topological
  monoid.

  The equivalence of \eqref{item:summary-1}, \eqref{item:summary-7},
  \eqref{item:summary-8}, and \eqref{item:summary-10} follows from
  Theorems~\ref{t:syntactic-clopen-in-compact-algebra}
  and~\ref{t:profinite-synt-congr}.
  To conclude the proof, it remains to observe that the implications
  \eqref{item:summary-8}$\,\Rightarrow\,$\eqref{item:summary-9} and
  \eqref{item:summary-9}$\,\Rightarrow\,$\eqref{item:summary-10} are
  trivial.
\end{proof}

Some of the results of Section~\ref{sec:cdsc} suggest looking at
locally compact residually discrete algebras as a generalization of
profinite algebras and, more generally at locally compact
0-dimensional algebras as a generalization of Stone topological
algebras. However, it is not clear where such a study might lead,
perhaps for lack of interesting examples. Examples~\ref{eg:cd-not-fd}
and~\ref{eg:counterexample-converse-lc-compactly-determined} show that
much of the good behavior observed in the compact case breaks down for
locally compact algebras.

\bibliographystyle{amsplain}
\bibliography{sgpabb,ref-sgps}

\end{document}